\documentclass[12 pt]{article}
\usepackage{stmaryrd}
\usepackage{amssymb}
\usepackage{amsmath}
\usepackage{amsthm}
\allowdisplaybreaks
\usepackage{amscd}
\usepackage{graphicx}
\usepackage{hyperref}
\usepackage{geometry}
\usepackage[all]{xy}
\usepackage{enumerate}
\usepackage{bookmark}

\usepackage{mathpazo}
\usepackage{avant}

\def\vbar{\mathchoice{\vrule height6.3ptdepth-.5ptwidth.8pt\kern- .8pt}
{\vrule height6.3ptdepth-.5ptwidth.8pt\kern-.8pt} {\vrule
height4.1ptdepth-.35ptwidth.6pt\kern-.6pt} {\vrule
height3.1ptdepth-.25ptwidth.5pt\kern-.5pt}}

\newtheorem{theorem}{Theorem}[section]

\newtheorem{lemma}[theorem]{Lemma}

\newtheorem{prop-def}[theorem]{Proposition-Definition}

\newtheorem{example}[theorem]{Example}

\begin{document}
\title{2-Local derivations on the W-algebra $W(2,2)$}
\author{\bf Xiaomin Tang}
\author{
Xiaomin Tang$^{1,2,}$
\footnote{Corresponding author,  E-mail: tangxm@hlju.edu.cn}
\\
{\small 1.   School of  Mathematical Science, Heilongjiang University, Harbin, 150080, P. R. China, } \\
{\small 2.  School of  Mathematical Science, Harbin Engineering University, Harbin, 150001, P. R. China}
}
\date{}
\maketitle

\begin{abstract}
The present paper is devoted to study 2-local derivations on W-algebra $W(2,2)$ which is an
infinite-dimensional Lie algebras with some out derivations. We prove that all 2-local derivations
on the W-algebra $W(2,2)$ are derivation. We also give a complete classification of the 2-local derivation
on  the so called thin Lie algebra and prove that it admits a lots of 2-local derivations which are not derivations.
\end{abstract}



\section{Introduction}

\ \ \ \ \ \ In 1997, \v{S}emrl \cite{Sem} introduced the notion of  2-local derivations
on algebras. Namely, for an associative algebra $\mathcal{L}$,  a map \(\Delta : \mathcal{L} \to
\mathcal{L}\) (not necessarily linear) is called a \textit{2-local
derivation} if, for every pair of elements \(x,y \in  \mathcal{L},\) there exists a
derivation \(\Delta_{x,y} : \mathcal{L} \to \mathcal{L}\) (depending on $x, y$) such that
\(\Delta_{x,y} (x) = \Delta(x)\) and \(\Delta_{x,y}(y) = \Delta(y).\)

The concept of 2-local derivation is actually an important and interesting property for an algebra.
For a given algebra \(\mathcal{L}\), the main problem concerning
these notions is to prove that they automatically become a
derivation or to give examples of 2-local derivations of \(\mathcal{L},\)
which are not derivations.
Recently, several papers have been devoted to similar notions and corresponding problems for  Lie algebras $\mathcal{L}$.
In \cite{AyuKudRak,ChenWang} the authors prove that every 2-local derivation on a semi-simple Lie algebra \(\mathcal{L}\) is a derivation and that
each finite-dimensional nilpotent Lie algebra, with dimension larger than two admits 2-local derivation which is not a
derivation.  In \cite{ayu2019} the authors study 2-local derivations on some infinite-dimensional Lie algebras, i.e., they that all 2-local derivations
on the Witt algebra as well as on the positive Witt algebra are (global) derivations, and give an example of infinite-dimensional Lie
algebra with a 2-local derivation which is not a derivation. In \cite{ayu2020,zhao2020} the authors prove that every 2-local derivation on some class of generalized Witt algebras
(or their Borel subalgebras) is a derivation.

 As we see that the Lie algebras whose every 2-local derivation is a derivation almost all have a common quality,
 that is any derivation of these Lie algebras is inner.  We naturally want to know what form of the 2-local derivation has if the Lie algebra has some out derivations?
  In the present paper we study 2-local derivations on the infinite-dimensional Lie algebra $W(2,2)$  and so called thin Lie algebra $\mathfrak{T}$.  Note that both $W(2,2)$ and $\mathfrak{T}$ all have some out derivations. We prove that every 2-local derivation on W-algebra $W(2,2)$ is a derivation and the tin Lie algebra $\mathfrak{T}$ admits many 2-local derivations which are not derivations.

In Section 2 we give some preliminaries concerning W-algebra $W(2,2)$. In Section 3
we prove that every 2-local derivations on W-algebra $W(2,2)$ are automatically
derivations. In Section 4 we complete describe the 2-local derivation on the so-called thin Lie algebra and show that it admits 2-local derivations which are not derivations.

Throughout this paper, we denote by $\mathbb{Z}$, $\mathbb{N}$, $\mathbb{Z}^*$ and $\mathbb{C}$ the sets of  all integers, positive integers, nonzero integers and complex numbers respectively. All algebras are over $\mathbb{C}$.

\medskip

\section{Preliminaries}

\medskip

\ \ \ \ \ \  In this section we give some necessary definitions and preliminary
results.

A derivation on a Lie algebra \(\mathcal{L}\) is a linear map
$D:\mathcal{L}\rightarrow \mathcal{L}$ which satisfies the Leibniz
law, that is,
$$
D([x,y])=[D(x),y]+[x, D(y)]
$$
for all $x,y\in \mathcal{L}.$ The set of all derivations of
\(\mathcal{L}\) with respect to the commutation operation is a Lie algebra and it
is denoted by $Der(\mathcal{L}).$ For all $a\in \mathcal{L}$,
the map ${\rm ad} (a)$ on \(\mathcal{L}\) defined as ${\rm ad} (a)x=[a,x],\ x\in\mathcal{L}$ is a derivation and derivations of this form are called
\textit{inner derivation}.

Recall that a map $\Delta: \mathcal{L}\rightarrow \mathcal{L}$
(not liner in general) is called a \textit{2-local derivation} if
for every $x,y\in \mathcal{L},$ there exists a derivation
$\Delta_{x,y}:\mathcal{L}\rightarrow \mathcal{L}$ (depending on $x,y$)
such that $\Delta(x)=\Delta_{x,y}(x)$ and $\Delta(x)=\Delta_{x,y}(y)$.
For a 2-local derivation on $\mathcal{L}$ and $k\in \mathbb{C}$, $x\in \mathcal{L}$, we have
\begin{equation}\label{recall1}
\Delta(kx)=\Delta_{x,kx}(kx)=k\Delta_{x,kx}(x)=k\Delta(x).
\end{equation}

The W-algebra $W(2,2)$ is an infinite-dimensional Lie algebra with the $\mathbb{C}$-basis $$\{L_m, I_m| m\in \mathbb{Z} \}$$ and
the Lie brackets are given by
\begin{eqnarray*}
&&[L_m,L_n]=(m-n)L_{m+n},\\
&&[L_m,I_n]=(m-n)I_{m+n}, \\
&&[I_m,I_n]=0, \ \forall m,n\in \mathbb{Z}.
\end{eqnarray*}
A class of central extensions of $W(2,2)$ first introduced by \cite{JiangDong} in their recent work on the classification of some simple vertex operator algebras, and then some scholars studied the theory on structures and representations of $W(2,2)$  or its central extensions, see \cite{Chenhj,GJP,Jiangzhangwei,Rad,tangw22,Wangy} and so forth. We now recall and establish  several auxiliary results.

\begin{lemma}\label{lemma_1} (see \cite{GJP})
Denote by ${\rm{Der}} (W(2, 2))$ and by ${\rm{Inn}} (W(2, 2))$ the space of derivations and the
space of inner derivations of $W(2, 2)$ respectively. Then
$${\rm{Der}} (W(2, 2))={\rm{Inn}} (W(2, 2))\oplus\mathbb{C}D,$$
where $D$ is an outer derivation defined by $D(L_m)=0$, $D(I_m)=I_m$ for all $m \in \mathbb{Z}$.
\end{lemma}

\begin{lemma}\label{lemma_2}
Let $\Delta$ be a 2-local derivation on the W-algebra $W(2,2)$. Then for every $x,y\in W(2,2)$, there exists a derivation
$\Delta_{x,y}$ of $W(2,2)$ such that $\Delta(x)=\Delta_{x,y}(x)$, $\Delta(y)=\Delta_{x,y}(y)$ and it can be written as
\begin{equation}\label{tangtang6}
\Delta_{x,y}={\rm ad} (\sum_{k\in \mathbb{Z}}\left(a_k(x,y)L_k+b_k(x,y)I_k\right))+\lambda(x,y)D
\end{equation}
where $\lambda, a_k, b_k (k\in \mathbb{Z})$ are complex-valued functions on $W(2,2)\times W(2,2)$ and $D$ is given by Lemma \ref{lemma_1}.
\end{lemma}

\begin{proof}
By Lemma \ref{lemma_1}, obviously the derivation $\Delta_{x,y}$ can be written as the form of  (\ref{tangtang6}).
\end{proof}

\section{2-Local derivations on  $W(2,2)$}

Now we shall give the main result concerning 2-local derivations on $W(2,2)$.

\begin{theorem}\label{thm-tang}
Every 2-local derivation
on  the W-algebra $W(2,2)$ is a derivation.
\end{theorem}
For the proof of this Theorem we need several Lemmas.
For a 2-local derivation $\Delta: W(2,2)\rightarrow W(2,2)$ and
$x,y\in \mathcal{L},$ below we always use the symbol $\Delta_{x,y}$ for the derivation of $W(2,2)$ satisfying
$\Delta(x)=\Delta_{x,y}(x)$ and $\Delta(x)=\Delta_{x,y}(y)$; and $D$ for the out derivation of $W(2,2)$ given by Lemma \ref{lemma_1}.

\begin{lemma}\label{mainlem}
Let $\Delta$ be a 2-local derivation on $W(2,2)$. Take any but fixed $y\in W(2,2)$.
\begin{enumerate}[(i)]
\item  For a given $i\in \mathbb{Z}$, if  $\Delta (L_i)=0$  then
\begin{equation}\label{txm1}
\Delta_{L_i,y}={\rm ad} \left(a_i(L_i,y)L_i+b_i(L_i,y)I_i\right)+\lambda(L_i,y)D;
\end{equation}

\item If $\Delta (I_0)=0$  then for any $y\in W(2,2)$  we have
\begin{equation}\label{txm2}
\Delta_{I_0,y}={\rm ad} (a_0(I_0,y)L_0+\sum_{k\in \mathbb{Z}} b_k(I_0,y)I_k)
\end{equation}
where $\lambda, a_k, b_k (k \in \mathbb{Z})$ are complex-valued functions on $W(2,2)\times W(2,2)$.
\end{enumerate}
\end{lemma}

\begin{proof}
 By Lemma \ref{lemma_2}, we can assume that
\begin{eqnarray}
\Delta_{L_i,y}={\rm ad} (\sum_{k\in \mathbb{Z}}(a_k(L_i,y)L_k+b_k(L_i,y)I_k))+\lambda(L_i,y)D, \label{zhong2tang}\\
\Delta_{I_0,y}={\rm ad} (\sum_{k\in \mathbb{Z}}(a_k(I_0,y)L_k+b_k(I_0,y)I_k))+\lambda(I_0,y)D \label{zhong3tang}
\end{eqnarray}
for some  complex-valued functions $\lambda, a_k, b_k (k \in \mathbb{Z})$  on $W(2,2)\times W(2,2)$.

(i) When $\Delta(L_{i})=0$, in view of (\ref{zhong2tang}) we obtain
\begin{eqnarray*}
\Delta(L_{i})&=&\Delta_{L_i, y}(L_{i})\\
             &=&[\sum_{k\in \mathbb{Z}}(a_k(L_i,y)L_k+b_k(L_i,y)I_k), L_i]+\lambda(L_i,y)D(L_i)\\
             &=& \sum_{k\in \mathbb{Z}}((k-i)a_k(L_i,y)L_{k+i}+(k-i)b_{k}(L_i,y)I_{k+i})=0.
\end{eqnarray*}
From the above equation, one has
$(k-i)a_k(L_i,y)=(k-i)b_{k}(L_i,y)=0$ for all $k\in \mathbb{Z}$, which deduces $a_k(L_i,y)=b_k(L_i,y)=0$ for all $i\in \mathbb{Z}$ with $k\neq i$. Then
Equation (\ref{zhong2tang}) becomes (\ref{txm1}), as deserved.

(ii) When $\Delta(I_{0})=0$, then it follows from (\ref{zhong3tang}) that
\begin{eqnarray*}
\Delta(I_{0})&=&\Delta_{I_0, y}(I_{0})\\
             &=&[\sum_{k\in \mathbb{Z}}(a_k(I_0, y)L_k+b_k(I_0, y)I_k), I_0]+\lambda(I_0, y)D(I_0)\\
             &=& \sum_{k\in \mathbb{Z}}ka_k(I_0, y)L_{k}+\lambda(I_0, y)I_0=0.
\end{eqnarray*}
Then we have $\lambda(I_0, y)=0$ and $ka_k(I_0, y)=0$ for all $k\in \mathbb{Z}$, i.e.,  $a_k(I_0, y)=0$ for all $k\in \mathbb{Z}^*$. This with (\ref{zhong3tang}) implies that (\ref{txm2})  holds. The proof is completed.
\end{proof}

\begin{lemma}\label{lem13}
Let $\Delta$ be a 2-local derivation on $W(2,2)$ such that $\Delta(L_{0})=\Delta(L_{1})=0.$
Then
\begin{eqnarray}\label{zhong1}
\Delta(L_{i})=0, \ \ \forall i\in \mathbb{Z}.
\end{eqnarray}
\end{lemma}

\begin{proof} In view of $\Delta(L_{0})=\Delta(L_{1})=0$, by using Lemma \ref{mainlem} we can assume that
\begin{eqnarray}
\Delta_{L_0,y}={\rm ad} (a_0(L_0,y)L_0+b_0(L_0,y)I_0)+\lambda(L_0,y)D, \label{zhong21}\\
\Delta_{L_1,y}={\rm ad} (a_1(L_1,y)L_0+b_1(L_1,y)I_1)+\lambda(L_1,y)D  \label{zhong22}
\end{eqnarray}
for all $y\in W(2,2)$, where $\lambda, a_k, b_k (k \in \mathbb{Z})$ are complex-valued functions  on $W(2,2)\times W(2,2)$.
Let $i\in\mathbb{Z}$ be a fixed index. Then by taking $y=L_i$ in  (\ref{zhong21}) and (\ref{zhong22}) respectively we get
\begin{eqnarray*}
\Delta(L_i)&=&\Delta_{L_0,L_i}(L_i)=[a_0(L_0,L_i)L_0+b_0(L_0,L_i)I_0, L_i]+\lambda(L_0,L_i)D(L_i),\\
   &=& -ia_0(L_0,L_i)L_i-ib_0(L_0,L_i)I_i
\end{eqnarray*}
and
\begin{eqnarray*}
\Delta(L_i)&=&\Delta_{L_1,L_i}(L_i)=[a_1(L_1,L_i)L_1+b_1(L_1,L_i)I_1, L_i]+\lambda(L_1,L_i)D(L_i),\\
   &=& (1-i)a_0(L_0,L_i)L_{i+1}+(1-i)b_0(L_0,L_i)I_{i+1}.
\end{eqnarray*}
By the above two equations, it follows that $$ia_0(L_0,L_i)L_i+ib_0(L_0,L_i)I_i+(1-i)a_0(L_0,L_i)L_{i+1}+(1-i)b_0(L_0,L_i)I_{i+1}=0,$$
which implies $a_0(L_0,L_i)=b_0(L_0,L_i)=0$.  It concludes that $\Delta(L_i)=0$.
We finish the proof.
\end{proof}

\begin{lemma}\label{lem16}
Let $\Delta$ be a 2-local derivation on $W(2,2)$ such that $\Delta(L_{i})=0$ for all $i\in \mathbb{Z}$. Then
for any $x=\sum_{t\in \mathbb{Z}} (\alpha_tL_t+\beta_t I_t)\in W(2,2)$, we have
\begin{equation}\label{guol}
\Delta(x)=\Delta(\sum_{k\in \mathbb{Z}} (\alpha_tL_t+\beta_t I_t))= \mu_x \sum_{t\in \mathbb{Z}} \beta_t I_t
\end{equation}
where $\mu_x$ is a complex number depending on $x$.
\end{lemma}

\begin{proof}
For $x=\sum_{t\in \mathbb{Z}} (\alpha_tL_t+\beta_t I_t)\in W(2,2)$, since $\Delta (L_i)=0$ for any $i\in \mathbb{Z}$, from Lemma \ref{mainlem} we have
\begin{eqnarray*}
\Delta(x)&=&\Delta_{L_i,x}(x)\\
&=&[a_i(L_i,x)L_i+b_i(L_i,x)I_i, x]+\lambda(L_i,x)D(x)\\
&=&\sum_{t\in \mathbb{Z}}(i-t)(\alpha_t a_i(L_i,x)L_{i+t}+(\beta_t a_i(L_i,x)+\alpha_t b_i(L_i,x)) I_{i+t} )+ \lambda(L_i,x) \sum_{t\in \mathbb{Z}} \beta_t I_t.
\end{eqnarray*}
By taking enough diffident $i\in \mathbb{Z}$ in the above equation and, if necessary, let these $i$'s to be large enough, we obtain that
$\Delta(x)=\lambda(L_i,x) \sum_{t\in \mathbb{Z}} \beta_t I_t.$ Note that $\mu_x\doteq\lambda(L_i,x)$ is a constant since it is independent on $i$.
\end{proof}

\begin{lemma}\label{tangtang666}
Let $\Delta$ be a 2-local derivation on $W(2,2)$ such that $\Delta(I_{0})=0$ and $\Delta(L_{i})=0$ for all $i\in \mathbb{Z}$.
Then for any $p\in \mathbb{Z}^*$ and $y\in W(2,2)$, there are  $\xi_p^y,\eta_p^y\in \mathbb{C}$ such that
\begin{eqnarray}
\Delta_{L_{p}+I_{2p}, y}= {\rm ad} (\xi_p^yL_p+\eta_p^y I_p+ \xi_p^y I_{2p}).\label{xi2}
\end{eqnarray}
\end{lemma}

\begin{proof}
For $p\in \mathbb{Z}^*$, by $\Delta(L_{i})=0$ for all $i\in \mathbb{Z}$ and Lemma \ref{lem16} we have
\begin{equation}\label{tXXz}
\Delta(L_{p}+I_{2p})=\mu_{L_{p}+I_{2p}} I_{2p},
\end{equation}
where $\mu_{L_{p}+I_{2p}}\in \mathbb{C}$ is given by (\ref{guol}).  In view of $\Delta(I_{0})=0$ and
Lemma \ref{txm1} we know that
\begin{eqnarray*}
&&\Delta(L_{p}+I_{2p})\\
&=&\Delta_{I_0,L_{p}+I_{2p}}(L_{p}+I_{2p})\\
&=&[a_0(I_0,L_{p}+I_{2p})L_0+\sum_{k\in \mathbb{Z}} b_k(I_0,L_{p}+I_{2p})I_k, L_{p}+I_{2p}]\\
&=& -pa_0(I_0,L_{p}+I_{2p})(L_{p}+2I_{2p})+\sum_{k\in \mathbb{Z}}(k-p)b_k(I_0,L_{p}+I_{2p})I_{k+p}.
\end{eqnarray*}
This, together with (\ref{tXXz}), gives that $-pa_0(I_0,L_{p}+I_{2p})=0$ and $-2pa_0(I_0,L_{p}+I_{2p})=\mu_{L_{p}+I_{2p}}$, i.e., we get  $\mu_{L_{p}+I_{2p}}=0$.
It follows by (\ref{tXXz}) that
\begin{equation}
\Delta(L_{p}+I_{2p})=0. \label{xi1}
\end{equation}
Next, for every $y\in W(2,2)$, by Lemma \ref{lemma_2} we can assume that
\begin{equation}\label{tangtang66}
\Delta_{L_p+I_{2p},y}={\rm ad} (\sum_{k\in \mathbb{Z}}(a_k(L_p+I_{2p},y)L_k+b_k(L_p+I_{2p},y)I_k))+\lambda(L_p+I_{2p},y)D.
\end{equation}
From (\ref{xi1}) and (\ref{tangtang66}), one has
\begin{eqnarray*}
&&\Delta(L_{p}+I_{2p})\\
&=&\Delta_{L_p+I_{2p},y}(L_{p}+I_{2p})\\
&=&[\sum_{k\in \mathbb{Z}}(a_k(L_p+I_{2p},y)L_k+b_k(L_p+I_{2p},y)I_k), L_{p}+I_{2p}]+\lambda(L_p+I_{2p},y)I_{2p}\\
&=&\sum_{k\in \mathbb{Z}}a_k(L_p+I_{2p},y)((k-p)L_{k+p}+(k-2p)I_{k+2p})\\
&&\ \ \ \ +\sum_{k\in \mathbb{Z}}(k-p)b_k(L_p+I_{2p},y)I_{k+p}+\lambda(L_p+I_{2p},y)I_{2p}=0.
\end{eqnarray*}
From this, it is easy to see that $(k-p)a_k(L_p+I_{2p},y)L_{k+p}=0$ for all $k\in \mathbb{Z}$ and so that $a_k(L_p+I_{2p},y)=0$ for
all $k\neq p$. Using this conclusion we observe the coefficient of $I_{3p}$ in the above equation, then one has
$$(p-2p)a_p(L_p+I_{2p},y)+(2p-p)b_{2p}(L_p+I_{2p},y)=0,$$
which implies $a_p(L_p+I_{2p},y)=b_{2p}(L_p+I_{2p},y)$. Furthermore,  by observing the coefficient of $I_{k}, k\neq 3p$ in the above equation
we get $\lambda(L_p+I_{2p},y)=0$ and $(k-p)b_k(L_p+I_{2p},y)=0$ for all $k\neq p, 2p$, i.e., $b_k(L_p+I_{2p},y)=0$ for all $k\neq p, 2p$.
Finally, by denoting $\xi_p^y=a_p(L_p+I_{2p},y)$ and $\eta_p^y=b_p(L_p+I_{2p},y)$ we finish the proof.
\end{proof}

\begin{lemma}\label{lem-mel6}
Let $\Delta$ be a 2-local derivation on $W(2,2)$ such that $\Delta(L_{0})=\Delta(L_{1})=\Delta(I_{0})=0.$ Then $\Delta(x)= 0$ for all $x\in W(2,2)$.
\end{lemma}

\begin{proof}

Take any but fixed $x=\sum_{t\in \mathbb{Z}} (\alpha_tL_t+\beta_t I_t)\in W(2,2)$, where $(\alpha_t)_{t\in \mathbb{Z}}, (\beta_t)_{t\in \mathbb{Z}}$ are
both sequences which contain only finitely many nonzero entries.

Since $\Delta(L_{0})=\Delta(L_{1})=0$, it follows by Lemma \ref{lem13} that
\begin{eqnarray}\label{zhong166}
\Delta(L_{i})=0, \ \ \forall i\in \mathbb{Z}.
\end{eqnarray}
This, together with Lemma \ref{lem16}, gives
\begin{equation}\label{guol6}
\Delta(x)=\Delta(\sum_{k\in \mathbb{Z}} (\alpha_tL_t+\beta_t I_t))= \mu_x \sum_{t\in \mathbb{Z}} \beta_t I_t
\end{equation}
for some $\mu_x\in \mathbb{C}$. Now, for any $p\in \mathbb{Z}^*$, by (\ref{zhong166}) and
$\Delta(I_{0})=0$, we obtain by Lemma \ref{tangtang666} that
\begin{eqnarray}
\Delta_{L_{p}+I_{2p}, x}= {\rm ad} (\xi_p^xL_p+\eta_p^x I_p+ \xi_p^x I_{2p})\label{xi26}
\end{eqnarray}
for some  $\xi_p^x,\eta_p^x\in \mathbb{C}$. Therefore, from (\ref{xi26}) one has
\begin{eqnarray}
\Delta(x)&=&\Delta_{L_p+I_{2p},x}(x)\nonumber\\
&=&[\xi_p^xL_p+\eta_p^x I_p+ \xi_p^x I_{2p}, \sum_{t\in \mathbb{Z}} (\alpha_tL_t+\beta_t I_t)]\nonumber\\
&=&\sum_{t\in \mathbb{Z}}((p-t)\xi_p^x \alpha_tL_{p+t}+(p-t)\xi_p^x \beta_tI_{p+t}) \label{final}\\
&&\ \ \ \ +\sum_{t\in \mathbb{Z}}((p-t)\eta_p^x \alpha_t I_{p+t}+(2p-t)\xi_p^x \alpha_t I_{2p+t}.\nonumber
\end{eqnarray}
Next the proof is divided into three cases according to the situations of $(\alpha_t)_{t\in \mathbb{Z}}, (\beta_t)_{t\in \mathbb{Z}}$.

{\bf Case i. } $(\beta_t)_{t\in \mathbb{Z}}$ is a zero sequence, i.e., $x=\sum_{t\in \mathbb{Z}} \alpha_tL_t$. Then by (\ref{guol6}), it is easy to see that $\Delta(x)=0$.

{\bf Case ii. } $(\alpha_t)_{t\in \mathbb{Z}}$ is a zero sequence, i.e., $x=\sum_{t\in \mathbb{Z}} \beta_tI_t$. Then by (\ref{guol6}) and (\ref{final}) we have
$$
\Delta(x)= \mu_x \sum_{t\in \mathbb{Z}} \beta_t I_t=\sum_{t\in \mathbb{Z}}(p-t)\xi_p^x \beta_tI_{p+t}
$$
for all $p\in \mathbb{Z}$. By taking enough diffident $p$ in the above equation and, if necessary, let these $p$'s to be large enough, we obtain that
 $\Delta(x)=0$.

{\bf Case iii.} Both $(\alpha_t)_{t\in \mathbb{Z}}$ and $(\beta_t)_{t\in \mathbb{Z}}$ are not zero sequences. Hence there is a nonzero term  $\alpha_{t_0}L_{t_0}$
in $x=\sum_{t\in \mathbb{Z}} (\alpha_tL_t+\beta_t I_t)$ for some $t_0\in \mathbb{Z}$.  Take two integers $p=p_1$ and $p=p_2$ in (\ref{final}) such that
$p_i-t_0\neq 0, i=1,2$, then by $(p_i-t_0)\xi_{p_i}^x \alpha_{t_0}L_{p_i+t_0}=0$ in (\ref{final}) we have $\xi_{p_i}^x=0$.
Then by (\ref{guol6}) and (\ref{final}) we have
$$
\Delta(x)= \mu_x \sum_{t\in \mathbb{Z}} \beta_t I_t=\sum_{t\in \mathbb{Z}}(p_i-t)\eta_{p_i}^x \alpha_t I_{p_i+t}, \ \  i=1,2.
$$
By taking $p_1$ and $p_2$ in the above equation such that $p_1, p_2, p_1-p_2$ are large enough, we see that $\Delta(x)=0$. The proof is completed.
\end{proof}

Now we are in position to prove Theorem \ref{thm-tang}.

\textbf{Proof of Theorem \ref{thm-tang} :}
Let $\Delta$ be a 2-local derivation on $W(2,2)$.
Take a derivation $\Delta_{L_0,L_1}$ such that
\begin{equation*}
\Delta(L_0)=\Delta_{L_0,L_1}(L_0)\ \ \text{and} \ \ \Delta(L_1)=\Delta_{L_0,L_1}(L_1).
\end{equation*}
Set $\Delta_1=\Delta-\Delta_{L_0,L_1}.$ Then $\Delta_1$ is a 2-local
derivation such that $\Delta_1(L_0)=\Delta_1(L_1)=0.$ By lemma \ref{lem13}, $\Delta_1(L_i)=0$ for all $i\in\mathbb{Z}.$
From this with Lemma \ref{lem16}, we have $\Delta_1(I_0)=\mu_{I_0}I_0$ for some $\mu_{I_0}\in \mathbb{C}$.
Now we set $\Delta_2=\Delta_1-\mu_{I_0}D$. Then $\Delta_2$ is a 2-local derivation such that
\begin{eqnarray*}
\Delta_2(L_{0})=\Delta_1(L_{0})-\mu_{I_0}D(L_{0})=0-0=0,\\
 \Delta_2(L_{1})=\Delta_1(L_{1})-\mu_{I_0}D(L_{1})=0-0=0,\\
  \Delta_2(I_{0})=\Delta_1(I_{0})-\mu_{I_0}D(I_{0})=\mu_{I_0}I_{0}-\mu_{I_0}I_{0}=0.
\end{eqnarray*}
By lemma \ref{lem-mel6}, it follows that $\Delta_2=\Delta-\Delta_{L_0,L_1}-\mu_{I_0}D\equiv0.$ Thus $\Delta=\Delta_{L_0,L_1}+\mu_{I_0}D$ is a derivation. The proof is completed.
\hfill$\Box$

\section{2-local derivation on the thin Lie algebra}

Let us consider the following (see \cite{Kha}) so-called {\it thin Lie algebra} $\mathfrak{T}$
with a basis \(\{e_n: n\in \mathbb{N}\}\), which is defined by the
following table of multiplications of the basis elements:
 $$[e_1,e_n]=e_{n+1},\ \ \ n\geq 2,$$
 and other products of the basis elements being zero. In this section,  we study the 2-local derivation on  the thin Lie algebra and prove that it admits a lots of
2-local derivations which are not derivations. Recall that the authors in \cite{ayu2019} give a special example of 2-local derivations on  $\mathfrak{T}$. The following lemma is given by \cite{ayu2019} with a slight difference.
\begin{lemma}\label{thm12}
Any derivation $\delta$ on the algebra thin Lie algebra $\mathfrak{T}$ if of the form $\delta\doteq \delta_{\alpha,\beta}^{(n,m)}$ which satisfies
\begin{eqnarray}
&& \delta_{\alpha,\beta}^{(n,m)}(e_1)=\sum\limits_{i=1}^n\alpha_ie_i, \label{a231}\\
&& \delta_{\alpha,\beta}^{(n,m)}(e_j)=(j-2)\alpha_1e_j+\sum\limits_{i=2}^m\beta_{i}e_{i+j-2},\ \ \ j\geq2, \label{a232}
\end{eqnarray}
where  $n,m-1\in \mathbb{N}$ and $\alpha=(\alpha_1, \cdots, \alpha_n)\in \mathbb{C}^n, \beta=(\beta_2, \cdots, \beta_m)\in \mathbb{C}^{m-1}$.
\end{lemma}

\begin{proof}  Let $\delta$ be a derivation on $\mathcal{L}.$  We set
$\delta(e_1)=\sum\limits_{i=1}^n\alpha_ie_i,\ \ \delta(e_2)=\sum\limits_{i=1}^m\beta_ie_i,$
where $\alpha_i, \beta_j\in \mathbb{C},$ $i=2, \cdots, n,\ j=1, \cdots,m $ and $n,m\in \mathbb{N}.$
Then we have $\delta(e_3)=\delta([e_1,e_2])=[\delta(e_1),e_2]+[e_1,\delta(e_2)]=\alpha_1e_3+\sum\limits_{i=1}^n\beta_{i}e_{i+1}$. From this, one has
$0=\delta([e_2,e_3])=[\delta(e_2),e_3]+[e_2, \delta(e_3)]=\beta_1e_4$, and so that $\beta_1=0$.
This means that (\ref{a232}) replacing $\delta_{\alpha,\beta}^{(n,m)}$ by $\delta$ holds for $j=2$.
We assume that (\ref{a232}) holds for $j(\ge 2)$. Further,
We have
\begin{eqnarray*}
\delta(e_{j+1})&=&\delta([e_1,e_j])=[\delta(e_1),e_j]+[e_1,\delta(e_j)]\\
&=&[\sum\limits_{i=1}^n\alpha_ie_i,e_j]+[e_1, (j-2)\alpha_1e_j+\sum\limits_{i=2}^m\beta_{i}e_{i+j-2}]\\
&=&(j-1)\alpha_1e_{j+1}+\sum\limits_{i=2}^m\beta_{i}e_{i+j-1},
\end{eqnarray*}
which proves that (\ref{a232}) holds For $j+1$. By induction on $j$ we know that (\ref{a232}) holds.
Conversely, it is easy to check that a linear map $\delta$ on $\mathfrak{T}$ satisfying (\ref{a231}) and (\ref{a232}) is a derivation.
Denote this derivation $\delta$ by $\delta_{\alpha, \beta}^{(n,m)}$. The proof is completed.
\end{proof}

Now we give a complete classification of the 2-local derivation on $\mathfrak{T}$ as follow.

\begin{theorem}\label{th-exp}
Every 2-local derivation $\Delta$ on the thin Lie algebra $\mathfrak{T}$ is of the form
$$\Delta=\delta_{\alpha,\beta}^{(s,t)}+\Omega_{\theta, \lambda}^{(q,m)}$$
for some $s,t-1,m-1\in \mathbb{N}$,  $\lambda\in \mathbb{C}$,  and $\alpha=(\alpha_1, \cdots, \alpha_s)\in \mathbb{C}^s, \beta=(\beta_2, \cdots, \beta_t)\in \mathbb{C}^{t-1}$,  $\theta=(\theta_2, \cdots, \theta_m)\in \mathbb{C}^{m-1}$ and $q\in \{t\in \mathbb{Z}: t>2 \}$, where $\delta_{\alpha,\beta}^{(s,t)}$ is given by Lemma \ref{thm12} and $\Omega_{\theta, \lambda}^{(q,m)}:\mathfrak{T}\rightarrow \mathfrak{T}$ is a map that satisfies for any $x=\sum_{i=1}^p k_ie_i\in \mathfrak{T}$,
\begin{equation}\label{ex3}
\Omega_{\theta, \lambda}^{(q,m)}(x) = \begin{cases}
\sum\limits_{i=2}^p\sum\limits_{j=2}^mk_i\theta_{j}e_{i+j-2}, & \text{if $k_1\neq 0$,}\\
\lambda k_q e_q, & \text{if $x=k_q e_q$ for some $q$ with $2<q\le p$,}\\
0, & \text{ others}
\end{cases}
\end{equation}
\end{theorem}

\begin{proof}
Suppose that $\Delta$ is a 2-local derivation on the thin Lie algebra $\mathfrak{T}$. Let $\widetilde{\Delta}=\Delta-\Delta_{e_1,e_2}$.
Then $\widetilde{\Delta}$ is also a 2-local derivation on the thin Lie algebra $\mathfrak{T}$ satisfying $\widetilde{\Delta}(e_1)=\widetilde{\Delta}(e_2)$=0.

Take any but fixed $x=\sum\limits_{i=1}^pk_ie_i\in \mathfrak{T}$.  If $x=0$, then by (\ref{recall1}) we know $\widetilde{\Delta}(x)=0$. Hence below we always
assume that $x\neq 0$, i.e., $k_p\neq 0$ for some $p\in \mathbb{N}$.

For the derivation $\widetilde{\Delta}_{e_1, x}$, as $\widetilde{\Delta}_{e_1, x}(e_1)=\widetilde{\Delta}(e_1)=0$, it follows by Lemma \ref{thm12} that
\begin{eqnarray*}
&&\widetilde{\Delta}_{e_1, x}(e_1)=0, \\
&&\widetilde{\Delta}_{e_1, x}(e_j)=\sum\limits_{i=2}^m\beta_{i}^xe_{i+j-2} , \ \forall j\ge 2,
\end{eqnarray*}
for some $m\in \mathbb{N}$ with $m\ge 2$ and $\beta_{i}^x\in \mathbb{C}, i=2, \cdots, m$ with $\beta_{m}^x\neq 0$. Therefore we have
\begin{eqnarray}\label{ximage1}
\widetilde{\Delta}(x)&=&\widetilde{\Delta}_{e_1, x}(x)=k_1\nonumber\widetilde{\Delta}_{e_1, x}(e_1)+\cdots+k_p\nonumber\widetilde{\Delta}_{e_1, x}(e_p)\\
&=&k_2\beta_{2}^xe_2+(k_2\beta_{3}^x+k_3\beta_{2}^x)e_3 +\cdots \\
 &&\ \  +(k_{p-1}\beta_{m}^x+k_p\beta_{m-1}^{x})e_{p+m-3}+k_p\beta_{m}^xe_{p+m-2}\nonumber\\
 &=& \beta_{2}^x \sum\limits_{i=2}^pk_ie_i+\beta_{3}^x \sum\limits_{i=2}^pk_ie_{i+1}+\cdots+\beta_{m}^x \sum\limits_{i=2}^pk_ie_{i+m-2}\nonumber
\end{eqnarray}

For the derivation $\widetilde{\Delta}_{e_2, x}$, by $\widetilde{\Delta}_{e_2, x}(e_2)=\widetilde{\Delta}(e_2)=0$ and Lemma \ref{thm12}, we have
\begin{eqnarray*}
&&\widetilde{\Delta}_{e_2, x}(e_1)=\sum_{i=1}^n \alpha_i^x e_i, \\
&& \widetilde{\Delta}_{e_2, x}(e_j)=(j-2)\alpha_1^x e_j , \ \forall j\ge 2,
\end{eqnarray*}
for some $n\in \mathbb{N}$  and $\alpha_{i}^x\in \mathbb{C}, i=1, \cdots, n$ with $\alpha_{n}^x\neq 0$. From this, we get
\begin{eqnarray}\label{ximage2}
\widetilde{\Delta}(x)&=&\widetilde{\Delta}_{e_2, x}(x)\nonumber\\
&=&k_1(\alpha_{1}^xe_1+\cdots+\alpha_{n}^xe_n) \\
 &&\ \  +k_3\alpha_{1}^xe_3+2k_4\alpha_{1}^xe_4+\cdots+(p-2)k_p\alpha_{1}^xe_p.\nonumber
\end{eqnarray}
Next, according to the situations of coefficients $k_1, \cdots, k_p$ in $x=\sum\limits_{i=1}^pk_ie_i$, the proof is divided into the following cases.

{\bf Case 1.} When $k_1\neq 0$. By comparing (\ref{ximage1}) with (\ref{ximage2}), we have $k_1\alpha_{1}^xe_1=0$ and so that $\alpha_{1}^x=0$.
Therefore, (\ref{ximage2}) becomes
$$
\widetilde{\Delta}(x)=k_1(\alpha_{2}^xe_2+\cdots+\alpha_{n}^xe_n).
$$
This, together with (\ref{ximage1}), gives that $n=p+m-2$ and
\begin{equation}\label{tangzz}
\begin{cases}
k_1\alpha_2^x=&k_2\beta_2^x, \\
k_1\alpha_3^x=&k_2\beta_3^x+k_3\beta_2^x,\\
k_1\alpha_4^x=&k_2\beta_4^x+k_3\beta_3^x+k_4\beta_2^x,\\
\vdots &\vdots\\
k_1\alpha_n^x=&k_p\beta_m^x.
\end{cases}
\end{equation}
Note that $k_1\neq 0$,  if we given a sequence of numbers $\beta_2^x, \cdots, \beta_m^x$
then we can get a sequence of numbers $\alpha_2^x, \cdots, \alpha_n^x$ satisfying (\ref{tangzz}). Hence in this case
 we let $\Delta(x)$  be of the form (\ref{ximage1}), namely, by denoting $\theta_j=\beta_{j}^x, j=2, \cdots, m$ we have
$$
\widetilde{\Delta}(x)=\theta_{2} \sum\limits_{i=2}^pk_ie_i+\theta_{3}
\sum\limits_{i=2}^pk_ie_{i+1}+\cdots+\theta_{m} \sum\limits_{i=2}^pk_ie_{i+m-2}=\sum_{i=2}^p\sum_{j=2}^mk_i\theta_{j}e_{i+j-2}.
$$

{\bf Case 2.} When $k_1=0$. By (\ref{ximage2}) we have
\begin{equation}\label{zzz1}
\widetilde{\Delta}(x)=\alpha_{1}^x(k_3e_3+2k_4e_4+\cdots+(p-2)k_pe_p).
\end{equation}
From this we see that if $p=2$ or $\alpha_{1}^x =0$ then  $\widetilde{\Delta}(x)=0$. Assume that $p\ge 3$ and $\alpha_{1}^x \neq 0$. On the other hand, by (\ref{zzz1}) and (\ref{ximage1}) we see that  $\alpha_{1}^x(p-2)k_pe_p=k_p\beta_{m}^xe_{p+m-2}$ and so that $p=p+m-2$. In other words, $m=2$. Therefore, (\ref{ximage1}) becomes
\begin{equation}\label{zzz2}
\widetilde{\Delta}(x)=\beta_{2}^x (k_2e_2+\cdots+k_pe_p).
\end{equation}

{\bf Subcase 2.1} When $k_2\neq 0$. Then by (\ref{zzz1}) and (\ref{zzz2}) one has $\beta_{2}^x k_2e_2=0$ which deduces $\beta_{2}^x=0$.
Hence by (\ref{zzz2}) we have  $\widetilde{\Delta}(x)=0$.

{\bf Subcase 2.2} When $k_2= 0$. In view of (\ref{zzz1}) and (\ref{zzz2}), we get
\begin{equation}\label{complete1}
\widetilde{\Delta}(x)=\alpha_{1}^x(k_3e_3+2k_4e_4+\cdots+(p-2)k_pe_p)=\beta_{2}^x (k_3e_3+k_4e_4+\cdots+k_pe_p).
\end{equation}
If there are two coefficients $k_s, k_t$, $3\le s<t\le p$ in (\ref{complete1}) such that $k_s k_t\neq 0$,
then we have $\alpha_{1}^x(s-2)=\beta_{2}^x$ and $\alpha_{1}^x(t-2)=\beta_{2}^x$. This yields $\alpha_{1}^x=0$ and then $\widetilde{\Delta}(x)=0$.
If there exist only one $k_q\neq 0$ for some $3\le q \le p$, i.e., $x=k_q e_q$, then we have by (\ref{complete1}) that $\widetilde{\Delta}(x)=\lambda k_qe_q$ by denoting
$\lambda\doteq\beta_{2}^x$. If all $k_j's$ are equal to $0$, then $\widetilde{\Delta}(x)=0$.

Now, by summarizing the above processes we get $\widetilde{\Delta}=\Omega_{\theta, \lambda}^{(q,m)}$ for some appropriate $\theta, \lambda, q,m$. Note that $\widetilde{\Delta}=\Delta-\Delta_{e_1,e_2}$. Let the derivation $\Delta_{e_1,e_2}$ be of the form $\delta_{\alpha,\beta}^{(n,s)}$ for some appropriate $\alpha,\beta,n,s$ in
Lemma \ref{thm12}, then we complete the proof.
\end{proof}

By Theorem \ref{th-exp}, we know the thin Lie algebra admits a lots of
2-local derivations which are not derivations. We give two examples as follows.

\begin{example}
Let $\Delta=\delta_{\alpha,\beta}^{(s,t)}+\Omega_{\theta, \lambda}^{(q,m)}:\mathfrak{T}\rightarrow \mathfrak{T}$ with $m=2$ and
$\alpha=0$, $\beta=0$, $\theta=1$, $\lambda=0$, that is
\begin{equation*}
\Delta(\sum\limits_{i=1}^pk_ie_i) =
\begin{cases}
\sum\limits_{i=2}^pk_ie_i, & \text{if } k_1\neq 0,\\
0, & \text{if } k_1=0.
\end{cases}
\end{equation*}
The authors in \cite{ayu2019} have shown that such  $\Delta$ is a 2-local derivation on
\(\mathfrak{T}\) but it is not a derivation.
\end{example}

\begin{example}
Let $\delta_{\alpha,\beta}^{(s,t)}=\delta_{\alpha,\beta}^{(s,t)}+\Omega_{\theta, \lambda}^{(q,m)}:\mathfrak{T}\rightarrow \mathfrak{T}$
with $\delta_{\alpha,\beta}^{(s,t)}=0$, $m=q=3$ and
$\theta=(1,1)$, $\lambda=2$, that is $\Delta=\Omega_{(1,1), 2}^{(3,3)}$. Exactly we have
\begin{equation*}\label{yb4}
\Delta(\sum\limits_{i=1}^pk_ie_i) =
\begin{cases}
\sum\limits_{i=2}^pk_ie_i+\sum\limits_{i=2}^pk_ie_{i+1}, & \text{if } k_1\neq 0,\\
2k_3 e_3, & \text{if } \sum\limits_{i=1}^pk_ie_i=k_3 e_3,\\
0, & \text{if } k_1=0.
\end{cases}
\end{equation*}
Then by theorem \ref{th-exp} it is easy to see that $\Delta$ is a $2$-local derivation. We will see that $\Delta$ is not a derivation.
In fact, let $x=e_1+e_2$ and $y=-e_1-e_2+2e_3.$ Then we have $\Delta(x)=e_2+e_3$, $\Delta(y)=-e_2+e_3+2e_4$ and $\Delta(x+y)=\Delta(2e_3)=4e_3\neq \Delta(x)+\Delta(y)=2e_3+2e_4$.
So, \(\Delta\) is not
additive, and therefore is not a derivation.
\end{example}

\section*{Acknowledgments}
This work is supported in part by National Natural Science Foundation of China (Grant No. 11771069) and the fund of Heilongjiang Provincial Laboratory of the Theory and Computation of Complex Systems.

\end{document}